\documentclass[11pt]{article}
\usepackage{amsfonts,amsmath,amssymb,graphicx,color}
\usepackage{amsthm}
\usepackage{psfrag}
\usepackage{pdflscape}
\usepackage{multirow}
\usepackage{blindtext} 
\usepackage{enumitem}
\graphicspath{ {./figures/} }
\usepackage{hyperref}
\hypersetup{colorlinks=true}
\usepackage{algorithm}
\usepackage{algpseudocode}
\usepackage{epstopdf}
\usepackage{framed}
\usepackage{footnote}
\usepackage{geometry}
\usepackage{pifont}
\usepackage{color}
\usepackage{xfrac}
\usepackage{todonotes}
\usepackage{amssymb}
\usepackage{pifont}
\usepackage[numbers,sort&compress]{natbib}

\usepackage{lscape}

\newtheorem{thm}{Theorem}[section]

\newtheorem{cor}[thm]{Corollary}
\numberwithin{equation}{section}

\newcommand{\R}{\mathbb{R}}
\def\E{\mathbb{E}}

\begin{document}
\title{An Investigation of Newton-Sketch and Subsampled Newton Methods
}

\author{Albert S. Berahas\thanks{Department of Industrial and Systems Engineering, Lehigh University. \newline \url{albertberahas@gmail.com} This author was supported by the Defense Advanced Research Projects Agency (DARPA).}
       \and   
       Raghu Bollapragada\thanks{Department of Industrial Engineering and Management Sciences, Northwestern University. \newline\url{raghu.bollapragada@u.northwestern.edu} This author was supported by the Defense Advanced Research Projects Agency (DARPA).}
	   \and
       Jorge Nocedal\thanks{Department of Industrial Engineering and Management Sciences, Northwestern University. \newline\url{j-nocedal@northwestern.edu} This author was supported by the Office of Naval Research grant N00014-14-1-0313 P00003, and by National Science Foundation grant DMS-1620022}
            }

\maketitle

\begin{abstract}{

Sketching, a dimensionality reduction technique, has received much attention in the statistics community. In this paper, we study sketching in the context of Newton's method for solving finite-sum optimization problems in which the number of variables and data points are both large. We study two forms of sketching that perform dimensionality reduction in data space: Hessian subsampling and randomized Hadamard transformations. Each has its own advantages, and their relative tradeoffs have not been investigated in the optimization literature. Our study focuses on practical versions of the two methods in which the resulting linear systems of equations are solved approximately, at every iteration,  using an iterative solver. The advantages of using the conjugate gradient method vs. a stochastic gradient iteration are revealed through a set of numerical experiments, and a complexity analysis of the Hessian subsampling method is presented. 

}
\end{abstract}

\section{Introduction}
\label{sec:intro}
\setcounter{equation}{0}

In this paper, we consider second order methods for solving the finite sum problem,
\begin{align}  \label{problem}
\min_{w\in\mathbb{R}^d}F(w) =\frac{1}{n}\sum_{i=1}^{n}F_i(w),
\end{align}
in which each component function $F_i : \mathbb{R}^d \rightarrow \mathbb{R}$ is smooth and convex. Second order information is incorporated into the algorithms using sketching techniques, which construct an approximate Hessian by performing dimensionality reduction in  data space. Two forms of sketching are of  interest in the context of problem \eqref{problem}.  The simplest form,  referred to as  Hessian subsampling, selects individual Hessians at random and averages them. The second variant consists of sketching via randomized Hadamard transforms, which may yield a better approximation of the Hessian but can be computationally expensive. A study of the trade-offs between these two approaches is the subject of this paper. 

Our focus is on practical implementations designed for problems in which the number of variables $d$ and the number of data points $n$ are both large. A key ingredient is to avoid the inversion of large matrices by solving linear systems (inaccurately) at each iteration using an iterative linear solver. This leads us to the question of which is the most effective linear solver for the type of algorithms and problems studied in this paper.  In the light of the recent popularity of the stochastic gradient method, it is natural to investigate whether it could be preferable (as an iterative linear solver) to the conjugate gradient  method, which has been a workhorse in numerical linear algebra and optimization. On the one hand, a stochastic gradient iteration can exploit the stochastic nature of the Hessian approximations; on the other hand the conjugate gradient (CG) method enjoys attractive convergence properties. We present theoretical and computational results contrasting these two approaches and highlighting the superior properties of the conjugate gradient method. Given the appeal of the Hessian subsampling method that utilizes a CG iterative solver, we derive computational complexity bounds for it. 

A Newton method employing sketching techniques was proposed by  Pilanci and Wainwright \cite{pilanci2017newton} for problems in which the number of variables $d$ is not large but the number of data points $n$ can be extremely large. They assume that the objective function is such that a square root of the Hessian can be computed at reasonable cost (as is the case with generalized linear models). They analyze the relative advantages of various sketching techniques and provide computational complexity results under the assumption that linear systems are solved exactly. However, they do not present extensive numerical results, nor do they compare their approaches with Hessian subsampling. 

Second order optimization methods that employ Hessian subsampling techniques have been studied in 
 %
\cite{exact2018ima,drineas2011faster,erdogdu2015convergence,NIPS2016_6207,pilanci2017newton,roosta2018sub,wang2017sketching,wang2018sketched,woodruff2014sketching}.
Roosta-Khorasani and Mahoney \cite{roosta2018sub} derived convergence rates of subsampled Newton methods in probability for different sampling rates. Bollapragada et al.  \cite{exact2018ima} provided convergence and complexity results in expectation for strongly convex functions, and employed CG to solve the linear system inexactly. Recently, several works \cite{xu2017newton, xu2017second, yao2018inexact} have studied the theoretical and empirical performance of subsampled Newton methods for solving non-convex problems. Agarwal et al.  \cite{agarwal2016second} proposed to use a stochastic gradient-like method for solving the linear systems arising in subsampled Newton methods.

Given the  popularity of first order methods (see e.g.,  
\cite{sra2012optimization, bottou2017optimization,johnson2013accelerating,nesterov2013introductory,wright2015coordinate}
and the references therein), 
one may question whether algorithms that incorporate stochastic second order information have much to offer \emph{in practice} for solving large scale instances of problem \eqref{problem}.  In this paper, we argue that this is indeed the case in many important settings. Second order information can provide stability to the algorithm by facilitating the selection of the steplength parameter $\alpha_k$, and by  countering the effects of ill conditioning, without greatly increasing the cost of the iteration. 

The paper is organized as follows. In Section \ref{ns_ssn}, we describe sketching techniques for second order optimization methods. Practical implementations of the algorithms are presented in Section~\ref{sec:pract_methods}, and the results of our numerical investigation are given in Section~\ref{sec:num_study}. We present a convergence analysis for the Hessian subsampling method in Section \ref{sec:conv_analysis}, where we also compare its work complexity with that of Newton methods using other sketching techniques such as the Hadamard transforms.  Section~\ref{sec:main_findings} makes some final remarks about the contributions of the paper.

\section{Stochastic Second Order Methods}
\label{ns_ssn}

The methods studied in this paper are of the form,
\begin{equation}    \label{ssn}
A_k p_k = - \nabla F(w_k), \quad w_{k+1}= w_k + \alpha_k p_k,
\end{equation}
where $\alpha_k$ is the steplength parameter  and $A_k \succ 0$ is a stochastic approximation of the Hessian obtained via the sketching techniques described below. We show that a useful estimate of the Hessian can be obtained by reducing the dimensionality in  data space, making it possible to tackle problems with extremely large datasets.  We assume throughout the paper that the gradient $\nabla F(w_k)$ is exact, and therefore, the stochastic nature of the iteration is entirely due to the the choice of $A_k$. We could have employed a stochastic approximation to the gradient but opted not to do so because the use of an exact gradient allows us to study Hessian approximations in isolation and endows the algorithms with a $Q$-linear convergence rate, which makes them directly comparable with first order methods. 

 We now discuss how the Hessian approximations $A_k$ are constructed.

 \subsection*{Hessian-Sketch}

 As proposed by Pilanci and Wainwright \cite{pilanci2017newton}, the matrix $A_k$ can be defined as a randomized approximation to  the true Hessian  that  utilizes a decomposition of the matrix and approximates the Hessian via random projections in lower dimensions. Specifically, at every iteration, the algorithm defines a random \emph{sketch} matrix $S^k \in \mathbb{R}^{m \times n}$ with the property that $\mathbb{E}\left[{(S^k)^TS^k}/{m}\right]=I_n$, and  $m <n$. Next, the algorithm computes a square-root decomposition  of the  Hessian $\nabla^2 F(w_k)$, which we denote by $\nabla^2 F(w_k)^{\sfrac{1}{2}} \in \mathbb{R}^{n \times d}$, and defines the Hessian approximation as
   \begin{equation}	\label{sketch}
 A_k = \Big( (S^k\nabla^2 F(w_k)^{\sfrac{1}{2}})^TS^k\nabla^2 F(w_k)^{\sfrac{1}{2}}\Big) .
 \end{equation}
 The search direction of the algorithm is obtained by solving the linear system  in \eqref{ssn}. When $m << n$ forming $A_k$ is much less expensive than forming the Hessian $\nabla^2 F(w_k)$.

This approach is not always viable because a square root Hessian matrix may not be easily computable,  but there are important classes of problems when it is. Consider for example, generalized linear models where $F(w) = \sum_{i=1}^n g_i(w;x_i)$. The square-root matrix can be expressed as $ \nabla^2 F(w)^{\sfrac{1}{2}}= D^{\sfrac{1}{2}}X$, where $X \in \mathbb{R}^{n \times d}$ is the  data matrix, $x_i \in \mathbb{R}^{1\times d}$ denotes the $i$-th row of $X$, and $D$ is a diagonal matrix given by $D = {\rm diag}\{ g_i''(w;x_i)^{\sfrac{1}{2}}\}_{i=1}^n$. The multiplication of the sketch matrix $S^k$ with $ \nabla^2 F(w)^{\sfrac{1}{2}}$ leads to the projection of the latter matrix on to a smaller dimensional subspace. 

Randomized projections can be performed based on the fast Johnson-Lindenstrauss (JL) transform, e.g., Hadamard matrices, which allow for a fast matrix-matrix multiplication and have been shown to be efficient in least squares settings. We follow \cite{krahmer2011new, pilanci2017newton} and define the Hadamard matrices $H \in \mathbb{R}^{n\times n}$  as having entries $|H_{i,j}| \in \left[\frac{-1}{\sqrt{n}}, \frac{1}{\sqrt{n}}\right]$. The randomized Hadamard projections are performed using the sketch matrix $S^k\in \mathbb{R}^{m \times n}$ whose rows are chosen by randomly sampling $m$ rows of the matrix $\sqrt{n} HD$, where $D \in \mathbb{R}^{n \times n}$ is a diagonal matrix with random $\pm{1}$ entries.  
The properties of this approach have been analyzed in \cite{pilanci2017newton}, where it is shown that each iteration has complexity $\mathcal{O}(nd \log(m) +dm^2)$, which is   lower than the complexity of Newton's method ($\mathcal{O}(nd^2)$), when $n>d$ and $m = \mathcal{O}(d)$.

 \subsection*{Hessian Subsampling}
Hessian subsampling is a special case of sketching where the Hessian approximation is constructed by randomly sampling component Hessians $\nabla^2 F_i$ of \eqref{problem}, and averaging them. Specifically, at the $k$-th iteration, one defines
\begin{align}		\label{subh}
A_k = \nabla^2 F_{T_k}(w_k) = \frac{1}{\left| T_k \right|} \sum_{i \in T_k} \nabla^2 F_i(w_k) ,
\end{align}
where the index set $T_k \subset \{1,...,n \}$ is chosen either uniformly at random (with or without replacement) or in a non-uniform manner as described in \cite{xu2016sub}.  
The search direction is defined as the (inexact) solution of the system of equations in \eqref{ssn} with $A_k$ given by \eqref{subh}. By computing an approximate solution of the linear system that is only as accurate as needed to ensure progress toward the solution, one achieves computational savings, as discussed in the next section. Newton methods that use Hessian subsampling have recently received considerable attention, as mentioned above, and some of their theoretical properties are well understood \cite{exact2018ima,erdogdu2015convergence,roosta2018sub,xu2016sub}.  Although Hessian subsampling can be thought of as a special case of sketching,  we view it here as a distinct method because it can be motivated directly from the finite sum structure \eqref{problem} without referring to randomized projections.

\subsection*{Notation and Nomenclature} In the rest of the paper, we use the acronym SSN (Sub-Sampled-Newton) to denote algorithm \eqref{ssn} where $A_k$ is chosen by Hessian subsampling \eqref{subh}, and use the term \emph{Newton-Sketch} when other choices of the sketch matrix are employed, as in \eqref{sketch}.

\subsection*{Discussion}
We thus have two distinct classes of methods at our disposal. On the one hand,  Newton-Sketch enjoys superior statistical properties (as discussed below) but has a high computational cost. On the other hand, the SSN method  is simple to implement and is inexpensive, but could have inferior statistical properties.  For example, when individual component functions $F_i$ are highly dissimilar, taking a small sample of individual Hessians at each iteration of the SSN method may not yield a useful approximation to the true Hessian. For such problems, the Newton-Sketch method that forms the matrix based on a combination of all component functions may be more effective. 

In the rest of the paper, we compare the computational and theoretical properties of the two approaches when solving finite-sum problems involving many variables and large data sets.  In order to undertake this investigation we need to give careful consideration to the large-scale implementation of these methods.

\section{Practical Implementations of the Algorithms}
\label{sec:pract_methods}
 We can design efficient implementations of the Newton-Sketch and subsampled Hessian method for the case when the number of variables $d$ is very large by giving careful consideration to how the linear system in \eqref{ssn} is formed and solved. Pilanci and Wainwright \cite{pilanci2017newton}  proposed to implement the Newton-Sketch method using direct solvers, but this is only practical when $d$ is not  large. Since we do not wish to impose such a restriction on $d$, we utilize iterative linear solvers tailored to each specific method. We consider two iterative methods: the conjugate gradient method, and a stochastic gradient iteration \cite{agarwal2016second} that is inspired by (but it not equivalent to) the stochastic gradient method of Robbins and Monro \cite{RobMon51}.

\subsection{The Conjugate Gradient (CG) Method}
 For large scale problems, it can be very effective to employ a Hessian-free approach, and compute an inexact solution of linear systems with coefficient matrices \eqref{sketch} or  \eqref{subh} using an iterative method that only requires matrix-vector products instead of the matrix itself. The conjugate gradient method is the best known method in this category \cite{GoluvanL89}. One can compute a search direction $p_k$  by running $r \leq d$ iterations of CG on the system
\begin{align}	\label{sysa}
       A_k p_k = -\nabla F(w_k).
\end{align} 
 Under a favorable distribution of the eigenvalues of the matrix $A_k$, the CG method can generate a useful approximate solution in much fewer than $d$ steps; we exploit this result in the analysis presented in Section \ref{sec:conv_analysis}. We now discuss how to apply the CG method to the two classes of methods studied in this paper.

\subsubsection*{Newton-Sketch} 
The linear system \eqref{sketch}  arising in the Newton-Sketch method requires access to a square-root Hessian matrix.  This is a rectangular $n \times d$ matrix, denoted by $\nabla^2 F(w_k)^{\sfrac{1}{2}}$, such that  $[\nabla^2 F(w_k)^{\sfrac{1}{2}}]^T \nabla^2 F(w_k)^{\sfrac{1}{2}} =\nabla^2 F(w_k)$. Note that this is not the kind of decomposition (such as LU or QR) that facilitates the formation and solution of of linear systems involving the matrix \eqref{sketch}. Therefore, when $d$ is large it is imperative to use an iterative method such as CG. The entire matrix $A_k$ need not be formed, and  we only need to perform two matrix-vector products with respect to the sketched square-root Hessian. More specifically, for any vector $u \in \mathbb{R}^d$ one can compute 
\begin{align*}
v=\Big( (S\nabla^2 F(w)^{\sfrac{1}{2}})^TS\nabla^2 F(w)^{\sfrac{1}{2}}\Big) u
\end{align*}
 in two steps as follows
\begin{align}		\label{subNewton}
v_1=(S\nabla^2 F(w)^{\sfrac{1}{2}}) u, \qquad \text{and } \qquad v=(S\nabla^2 F(w)^{\sfrac{1}{2}})^Tv_1.
\end{align} 

Our implementation of the Newton-Sketch method includes two tuning parameters: the number of rows $m_{\text{NS}}$ in the sketch matrix $S$, and the maximum number of CG iterations, $\max_{\text{CG}}$, used to compute the step. We define the sketch matrix $S$ through the Hadamard basis that allows for a fast matrix-vector multiplication,  to obtain $S\nabla^2 F(w)^{\sfrac{1}{2}}$. This does not require forming the sketch matrix explicitly and can be executed at a cost of  $\mathcal{O}(m_{\text{NS}}d\log n)$ flops.  The maximum cost of every iteration is given as the evaluation of the gradient $\nabla F(w_k)$ plus $2m_{\text{NS}}\times max_{\text{CG}}$ component matrix-vector products, 
where the factor 2 arises due to 
\eqref{subNewton}.  We ignore the cost of forming the square-root matrix, which is not large in the case of generalized linear models; our test problems are of this form. We also ignore the cost of multiplying by the sketch matrix, which can be large. Therefore, the results we report in the next section are the \emph{most optimistic}  for Newton-Sketch.

\subsubsection*{Subsampled Newton}
The subsampled Hessian used in the SSN method is significantly simpler to define compared to sketched Hessian, and the CG solver is naturally suited to a Hessian-free approach. The product of $A_k$ times vectors can be computed by simply summing the individual  Hessian-vector products. That is, given a vector $u \in \mathbb{R}^d$,
	
\begin{equation}    \label{matv}
v= A_k u = \nabla^2F_{T_k}(w_k)u = \frac{1}{|T_k|}\sum_{i \in T_k}\nabla^2F_i(w_k)u .
\end{equation}	
We refer to this method as SSN-CG method. Unlike Newton-Sketch that requires the projection of the square-root matrix onto the lower dimensional manifold, the SSN-CG method can be implemented without the need of projections, and thus does not require additional computation beyond the matrix-vector products \eqref{matv}.

The tuning parameters involved in the implementation of this method are the maximum size of the subsample, $T$, and the maximum number of CG iterations used to compute the step, $\max_{\text{CG}}$. Thus, the maximum cost of a SSN-CG iteration is given by the cost of evaluating the gradient $\nabla F(w_k)$ plus $T\times \max_{\text{CG}}$ matrix-vector products \eqref{matv}.

\subsection{A Special Stochastic Gradient Iteration}
We can also compute an approximate solution of the linear system in \eqref{ssn} using a stochastic gradient iteration similar to that proposed in  \cite{agarwal2016second} and studied in \cite{exact2018ima}. This stochastic gradient iteration, that we refer to as SGI, is appropriate only in the context of the SSN method because for Newton-Sketch there is no efficient way to compute the stochastic gradient as the sketched Hessian cannot be decomposed into individual components.  The resulting SSN-SGI method operates in cycles of length $m_{\text{SGI}}$. At every iteration of the cycle, an index $i \in \{1,2,...,n \}$ is chosen at random, and a gradient step is computed for the following quadratic model at $w_k$:

\begin{align}	\label{eq:quad_component}
Q_i(p) = F(w_k) + \nabla F(w_k)^Tp + \frac{1}{2}p^T \nabla^2 F_i(w_k) p.
\end{align} 
This yields the iteration
\begin{align}		\label{eq:sgi_inner_alpha}
p_k^{t+1} = p^t_k - \bar \alpha\nabla Q_i(p^t_k) = (I - \bar \alpha\nabla^2 F_i(w_k))p^t_k  - \bar \alpha \nabla F(w_k), \quad \ 
\mbox{with} \ p^0_k = -\nabla F(w_k).
\end{align}
A method similar to SSN-SGI has been analyzed in \cite{agarwal2016second}, where it is referred to as LiSSA, and   \cite{exact2018ima} compares the complexity of LiSSA and SSN-CG. None of those papers report detailed numerical results, and the effectiveness of the SGI approach was not known to us at the outset of this investigation. 

The total cost of an (outer) iteration of the SSN-SGI method is given by the evaluation of the gradient $\nabla F(w_k)$  plus $m_{\text{SGI}}$  Hessian-vector products of the form $\nabla^2 F_i v$. This method has two tuning parameters: the inner iteration step length $\bar \alpha$ and the length $m_{\text{SGI}}$ of the inner SGI cycle.

\section{Numerical Study}
\label{sec:num_study}

In this section, we study the practical performance of the Newton-Sketch and Subsampled  Newton (SSN) methods described in the previous sections. As a benchmark, we  compare them against a first order method, and for this purpose we chose  SVRG \cite{johnson2013accelerating}, a method that has gained much popularity in recent years. The main goals of our numerical investigation are to determine whether the second order methods have advantages over first order methods, to identify the classes of problems for which this is the case, and to evaluate the relative strengths of the Newton-Sketch and SSN approaches. 
 
 We test two versions of the SSN method that differ in the choice of iterative linear solver, as mentioned above; we refer to them as SSN-CG and SSN-SGI. The steplength $\alpha_k$ in \eqref{ssn}  for the Newton-Sketch and the SSN methods was determined by an Armijo back-tracking line search, starting from a unit step length. For methods that use CG to solve the linear system, the CG iteration was terminated as soon as the following condition is satisfied,
\begin{align}	\label{residual}
\| \nabla ^2 F_{T_k}(w_k)p_k + \nabla F(w_k)\| < \zeta \| \nabla F(w_k) \|,
\end{align}
where $\zeta$ is a tuning parameter.  
For SSN-SGI, the SGI method was terminated after a fixed number,  $m_{\text{SGI}}$, of iterations.
The SVRG method operates in cycles. Each cycle begins with a full gradient step followed by $m_{\text{SVRG}}$ stochastic gradient-type iterations. SVRG has two tuning parameters: the (fixed) step length employed at every inner iteration and the length $m_{\text{SVRG}}$ of the cycle. The cost per cycle is the sum of a full gradient evaluation  plus $2m_{\text{SVRG}}$ evaluations of the component gradients $\nabla F_i$. Therefore, each cycle of SVRG is comparable in cost to the Newton-Sketch and SSN methods.

We test the methods on binary classification problems where the training  function is given by a logistic loss with $\ell_2$ regularization:
 \begin{align} 	\label{eq:bin_class}
  F(w) = \frac{1}{n}\sum_{i=1}^{n}\log(1+e^{-y^i(w^Tx^i)}) + \frac{\lambda}{2} \|w\|^2.
\end{align}
Here $(x^i, y^i)$,  $i=1, \ldots, n$, denote the training examples 
and $\lambda = \frac{1}{n}$.  For this objective function, the cost of computing the gradient is the same as the cost of computing a Hessian-vector product.
When reporting the numerical performance of the methods, we make use of this fact.  
The data sets employed in our experiments are listed in Appendix \ref{sec:ext_num}. They were selected to have different dimensions and condition numbers, and consist of both synthetic and public data sets.

\subsection*{Experimental Setup} To determine the best implementation of each method, we experimented with settings of parameters to achieve best practical performance. 
To do so, we first determined, for each method, the total amount of work to be performed between gradient evaluations. We call this the \emph{iteration budget} and denote it by $b$. For the second order methods,  $b$ controls the cost of forming and solving the linear system in \eqref{ssn}, while for the SVRG method it controls the amount of work of a complete cycle of length $m_{\text{SVRG}}$.  We tested 9\footnote{$b=\{n/100, n/50, n/10, n/5, n/2, n, 2n, 5n, 10n \}$} values of $b$ for each method and independently tested all possible combinations of the tuning parameters. 
For example, for the SVRG method, we set $b=n$, which implies that $m_{\text{SVRG}} = \frac{b}{2} = \frac{n}{2}$, and tuned the step length parameter $\alpha$. For Newton-Sketch, we chose pairs of parameters $(m_{\text{NS}},max_{\text{CG}})$ such that $\frac{m_{\text{NS}}\times max_{\text{CG}}}{2} = b$; for SSN-SGI, we set $m_{\text{SGI}} = b = n$, and tuned for the step length parameter $\alpha$; and for SSN-CG, we chose pairs of parameters $(T,max_{\text{CG}})$ such that $T\times max_{\text{CG}} = b$.

\subsection{Numerical Results} \label{results}

In Figure \ref{fig:best_performance}  we present a sample of results using four data sets (the rest of the results are given in Appendix~\ref{sec:ext_num}). We report training error ($F(w_k) - F^*$) vs. iterations,  as well as training error vs. effective gradient evaluations (defined as the sum of all function evaluations, gradient evaluations and Hessian-vector products). For SVRG one iteration denotes one  complete cycle. 

Our first observation, from the  second column of Figure \ref{fig:best_performance}, is that the SSN-SGI method is not competitive with the SSN-CG method. This was not easily predictable at the outset because the ability of the SGI method to use a new Hessian sample $\nabla^2 F_i$  at each iteration seems advantageous. But the results show that this is no match to the superior convergence properties of the CG method. Comparing with SVRG, we note that on the more challenging problems \texttt{gisette} and \texttt{australian} Newton-Sketch and SSN-CG are much faster; on problem \texttt{australian-scale}, which is well-conditioned, all methods perform on par; on the simple problem  \texttt{rcv1}  the benefits of using curvature information do not outweigh the increase in cost.   

The Newton-Sketch and SSN-CG methods perform on par on most of the problems in terms of effective gradient evaluations.  However, if we had reported CPU time, the SSN-CG method would show a dramatic advantage over Newton-Sketch. Interestingly, we observed that the optimal number of CG iterations  for the best performance of the two methods is very similar. On the other hand, the best sample sizes differ significantly: the optimal choice of the number of rows of the sketch matrix $m_{\text{NS}}$ in Newton-Sketch (using Hadamard matrices) is significantly lower than the number of subsamples $T$ used in the SSN-CG.  This is explained by the eigenvalue figures in Appendix~\ref{sec:eigs}, which show that a sketched Hessian with small number of rows is able to adequately capture the eigenvalues of the true Hessian on most test problems. 

More extensive results and commentary are provided in the Appendix.

\begin{figure}[H]  \label{represent}
	\begin{centering}
		\includegraphics[width=0.74\textwidth]{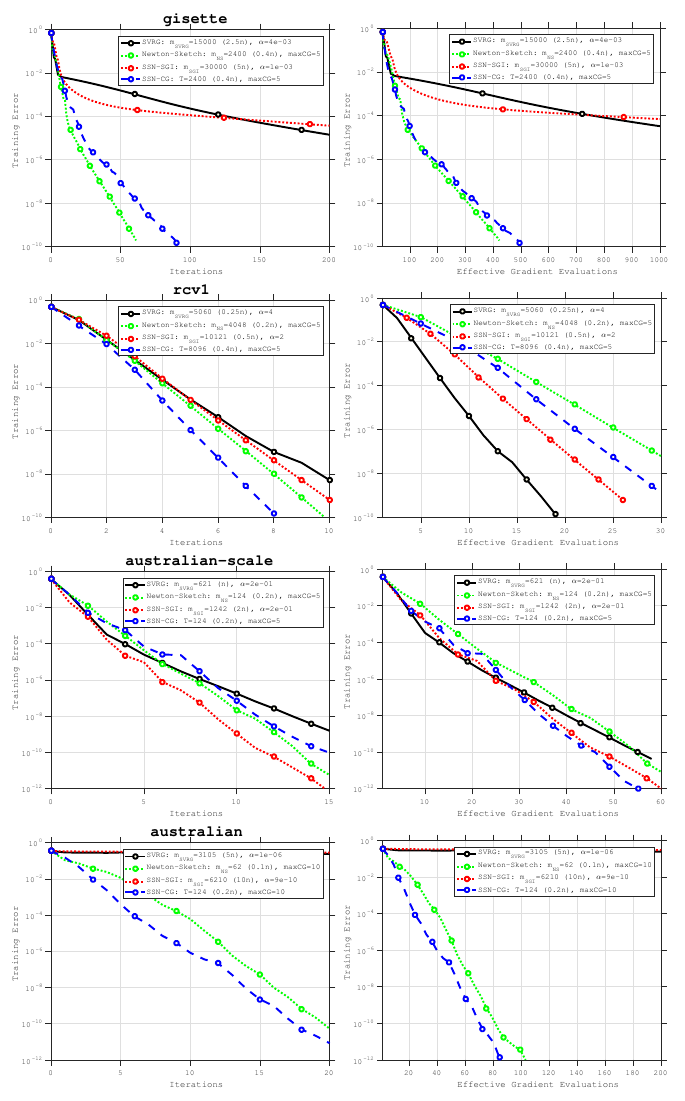}
	\par\end{centering}
	
	\caption{Comparison of SSN-SGI, SSN-CG, Newton-Sketch and SVRG on four machine learning data sets: \texttt{gisette}; \texttt{rcv1}; \texttt{australian\_scale}; \texttt{australian}. 
		Each row corresponds to a different problem. The first column reports training error ($F(w_k) - F^*$) vs iteration, and the second column reports training error vs effective gradient evaluations.  The tuning parameters of each method were chosen to achieve the best training error in terms of effective gradient evaluations.
	}
	\label{fig:best_performance}
\end{figure}

\section{Analysis}
\label{sec:conv_analysis}
We have seen in the previous section that the SSN-CG method is very efficient in practice. In this section, we discuss some of its theoretical properties.

A complexity analysis of the SSN-CG method based on the worst-case performance of  CG is given in \cite{exact2018ima}. That analysis, however, underestimates the power of the SSN-CG method on problems with favorable eigenvalue distributions and yields complexity estimates that are not indicative of its practical performance.  We now  present a more realistic analysis of  SSN-CG that exploits the optimal interpolatory properties of CG on convex quadratic functions. 

When applied to a symmetric positive definite linear system  $Ap=b$, the progress made by the CG method is not uniform but depends on the gap between eigenvalues. Specifically, if we denote the eigenvalues of $A$ by $0 < \lambda_1 \leq \lambda_2 \cdots \leq \lambda_d $, then the iterates $\{p^r\}$, $r=1,2, \ldots, d$, generated by the CG method satisfy \cite{Luenberger1984}   
\begin{align}  \label{cggap}
||p^{r} - p^*||_A^2 \leq \left(\frac{\lambda_{d-r+1} - \lambda_1}{\lambda_{d-r+1} + \lambda_1}\right)^2 \| p^0 - p^*||_A^2, \quad \mbox{where} \ \  A p^*=b.
\end{align}
Here $\| p\|^2_A = p^TAp.$
We make use of this property in the following local linear convergence result for the SSN-CG method, which is established under standard assumptions. 
{For simplicity}, we assume that the size of the sample $|T_k|=T $ used to define the Hessian approximation is constant throughout the run of the SSN-CG method (but the sample $T_k$ itself varies at every iteration).

\subsection*{Assumptions}

\begin{enumerate}
	\item [A.1] {\bf(Bounded Eigenvalues of Hessians)} Each function $F_i$ is twice continuously differentiable and each component Hessian is positive definite with eigenvalues lying in a positive interval. That is there exists positive constants $\mu$, $L$ such that
	\begin{equation} \label{eigen}
	\mu I   \preceq \nabla^2 F_{i}(w) \preceq L I, \quad \forall w \in \R^d.
	\end{equation}
	We define $\kappa = {L}/{\mu}$, which is an upper bound on the condition number of the true
	Hessian.
	\item [A.2] {\bf(Lipschitz Continuity of Hessian)} The Hessian of the objective function $F$ is Lipschitz continuous, 
	i.e., there is a constant $M >0$ such that 
	\begin{equation} \label{liphes}
	\|\nabla^2 F(w) - \nabla^2 F(z)\| \leq M\|w - z\|, \qquad \forall w,z  \in \R^d .
	\end{equation}
	\item [A.3] {\bf(Bounded Variance of Hessian components)} 
	There is a constant $\sigma$ such that, for all component Hessians, we have 
	\begin{equation}   \label{bvhes}
	\|\E[(\nabla^2 F_i(w) - \nabla^2 F(w))^2]\| \leq \sigma^2, \qquad \forall w \in \R^d.
	\end{equation}    
	\item [A.4] {\bf(Bounded Moments of Iterates)} There is a constant $\gamma > 0$ such that for any iterate $w_k$ generated by the SSN-CG method satisfies
	\begin{align}
	\E[\|w_k - w^* \|^2] \leq \gamma (\E[\| w_k - w^*\|])^2.
	\end{align}
\end{enumerate}

In practice, most finite-sum problems are regularized and therefore Assumption~A.1 is satisfied. Assumption~A.2 is a standard (and convenient) in the analysis of Newton-type methods. Concerning  Assumption~A.3, variances are always bounded for the finite-sum problem. Assumption~A.4 is imposed on non-negative numbers and is less restrictive than assuming that the iterates are bounded.

 Assumption A1 is relaxed in \cite{roosta2018sub} so that only the full Hessian $\nabla^2F$ is assumed positive definite, and not the Hessian of each individual  component,  $\nabla^2 F_i$.  However, \cite{roosta2018sub} only provides per-iteration improvement in probability, and hence, convergence of the overall method is not established. Rather than dealing with these intricacies, we rely on Assumption~A1 to focus instead on the interaction between the CG linear solver and the efficiency of method.

We now state the main theorem of the paper. We  follow the proof technique in \cite[Theorem 3.2]{exact2018ima} but rather than assuming the worst-case behavior of the CG method, we give a more realistic bound that exploits the properties of the CG iteration. Let, $0 < \lambda^{T_k}_1 \leq \lambda^{T_k}_2 \leq \cdots \leq \lambda^{T_k}_d$ denote the eigenvalues of subsampled Hessian $\nabla^2 F_{T_k}(w_k)$. In what follows, $\mathbb{E}$ denotes total expectation and $\mathbb{E}_k$ conditional expectation at the $k$th iteration.

\begin{thm} \label{thmain}
	Suppose that Assumptions A.1--A.4 hold.
	Let $\{w_k\}$ be the iterates generated by  the Subsampled Newton-CG method (SSN-CG) with $\alpha_k=1$, and
	\begin{align} \label{sampleb}
	|T_k|=T \geq \tfrac{64 \sigma^2 }{ \mu^2},
	\end{align}
	and suppose that the number of CG steps performed at iteration $k$ is the smallest integer $r_k$ such that 
	\begin{align}	\label{eig_cond}
	\left(\frac{\lambda_{d-r_k+1}^{T_k} - \lambda_1^{T_k}}{\lambda_{d-r_k+1}^{T_k} + \lambda_1^{T_k}}\right) \leq \frac{1}{8\kappa^{3/2}}.
	\end{align}
	Then, if $||w_0 - w^*|| \leq \min \left\{\frac{\mu}{2M}, \frac{\mu}{2\gamma M}\right\}$, we have
	\begin{equation}   \label{halfing}
	\E \left[||w_{k+1} - w^*||\right] \leq \tfrac{1}{2}\E \left[||w_k - w^*|| \right], \quad k=0,1,\ldots \,  .
	\end{equation}  
\end{thm}

\begin{proof}
	We write
	\begin{align}   
	\E_k[\|w_{k+1} - w^*\|] &= \E_k[\|w_k -w^* + p_k^r\|]   \nonumber \\
	& \leq \underbrace{\E_k[\|w_k - w^* - \nabla^2F_{T_k}^{-1}(w_k)\nabla F(w_k)\|]}_\text{Term 1} +
	\underbrace{\E_k[\|p_k^r + \nabla^2F_{T_k}^{-1}(w_k)\nabla F(w_k)\|]}_\text{Term 2} . \label{twott}
	\end{align} 
	Term 1 was analyzed in \cite{exact2018ima}, which establishes the bound
	\begin{align}	\label{twott1}
	\E_k[\|w_k - w^* - \nabla^2F_{T_k}^{-1}(w_k)\nabla F(w_k)\|] \leq & \frac{M}{2\mu}\|w_k - w^*\|^2  \nonumber\\
	& + \frac{1}{\mu}\E_k\left[\left\|\left(\nabla^2 F_{T_k}(w_k) - \nabla^2 F(w_k)\right)(w_k - w^*)\right\|\right].
	\end{align}
	Using the result in Lemma 2.3 from \cite{exact2018ima} and \eqref{sampleb}, we have 
	\begin{align}  
	\E_k[\|w_k - w^* - \nabla^2F_{T_k}^{-1}(w_k)\nabla F(w_k)\|] &\leq \frac{M}{2\mu} \|w_k - w^*\|^2 + \frac{\sigma}{\mu \sqrt{T}} \|w_k - w^*\|  \nonumber\\
	& \leq \frac{M}{2\mu} \|w_k - w^*\|^2 + \frac{1}{8} \|w_k - w^*\|. \label{term1}
	\end{align}
	
	Now, we analyze Term 2, which is the residual after $r_k$ iterations of the CG method. Assuming for simplicity that the initial CG iterate is $p^0_k=0$, we obtain from \eqref{cggap} 
	\begin{align*}
	\|p_k^r + \nabla^2F_{T_k}^{-1}\nabla F(w_k)\|_{A_k} &\leq \left(\frac{\lambda_{d-r_k+1}^{T_k} - \lambda_1^{T_k}}{\lambda_{d-r_k+1}^{T_k} + \lambda_1^{T_k}(w_k)}\right) \|\nabla^2F_{T_k}^{-1}\nabla F(w_k)\|_{A_k} ,
	\end{align*}
	where $A_k= \nabla^2F_{T_k}(w_k)$.
	To express this in terms of unweighted norms,  note that if $\| a \|^2_{A}  \leq \| b \|^2_{A} $, then
	\[  
	\lambda_1 \|a\|^2 \leq   a^T {A} a \leq   b^T {A} b \leq \lambda_d \|b\|^2 \ \Longrightarrow \ \|a\| \leq \sqrt{\kappa(A)} \|b\| \leq \sqrt{\kappa} \|b\|.
	\]
	Therefore, from Assumption A1 and due to the condition \eqref{eig_cond} on the number of CG iterations, we have
	\begin{align}
	\| p_k^r + \nabla^2F_{T_k}^{-1}(w_k)\nabla F(w_k) \| &\leq  \sqrt{\kappa}\left(\frac{\lambda_{d-r_k+1}^{T_k}(w_k) - \lambda_1^{T_k}(w_k)}{\lambda_{d-r_k+1}^{T_k}(w_k) + \lambda_1^{T_k}(w_k)}\right) \| \nabla^2F_{T_k}^{-1}(w_k)\nabla F(w_k) \| \nonumber \\
	&\leq \sqrt{\kappa}\frac{1}{8\kappa^{3/2}} \| \nabla F(w_k)\| \, \|\nabla^2F_{T_k}
	^{-1}(w_k)\| \nonumber \\
	&\leq \frac{L}{\mu}\frac{1}{8\kappa} \|w_k - w^*\| 
	\nonumber \\
	&= \frac{1}{8} \|w_k - w^*\| ,
	\label{eder}
	\end{align}  
	where the last equality follows from the definition $\kappa= {L}/{\mu}$.
	
	Combining \eqref{term1} and \eqref{eder}, we get 
	\begin{align}   
	\E_k[\|w_{k+1} - w^*\|]  	
	& \leq \frac{M}{2\mu} \|w_k - w^*\|^2 + \frac{1}{4} \|w_k - w^*\| .
	\end{align} 
	
	We now use induction to prove \eqref{halfing}. For the base case we have, by the assumption on $\| w_0\|$ ,
	\begin{align*}
	\E [\| w_1 - w^*\|] &\leq \frac{M}{2\mu} \|w_0 - w^*\|^2 + \frac{1}{4} \|w_0 - w^*\| \\
	& \leq \tfrac{1}{4} \|w_0 - w^*\| + \tfrac{1}{4}\|w_0 - w^*\| \\
	&= \tfrac{1}{2} \|w_0 - w^*\| .
	\end{align*}
	Now suppose that \eqref{halfing} is true for $k$-th iteration. Then, by Assumption~A.4
	\begin{align*}
	\E [\E_k[\|w_{k+1} - w^*\|]] &\leq \frac{M}{2\mu} \E[\|w_k - w^*\|^2] + \frac{1}{4} \E[\|w_k - w^*\|]\\
	&\leq \gamma \frac{M}{2\mu}\E[\|w_k - w^*\|] \, \E[\|w_k - w^*\|] +  \frac{1}{4} \E[\|w_k - w^*\|]\\
	& \leq \tfrac{1}{4}\E [\|w_k - w^*\|] + \tfrac{1}{4}\E[\|w_k - w^*\|] \\
	& \leq \tfrac{1}{2}\E[\|w_k - w^*\|] .
	\end{align*}
\end{proof}

Thus, if the initial iterate is near the solution $w^*$, then the SSN-CG method converges to $w^*$ at a linear rate, with a convergence constant of $1/2$. As we discuss below, when the spectrum of the Hessian approximation is favorable, condition \eqref{eig_cond} can be satisfied after a small number ($r_k$) of CG iterations. 


We now establish a  complexity result for the SSN-CG method that bounds the number of component gradient evaluations ($\nabla F_i$) and component Hessian-vector products ($\nabla^2 F_iv$) needed to obtain an $\epsilon$-accurate solution. In establishing this result we assume that the cost of the product $\nabla^2 F_iv$  is the same as the cost of evaluating a component gradient; this is the case in many practical applications such as the problem tested in the previous section. 


\begin{cor} \label{cgcom} Suppose that Assumptions A.1--A.4 hold, and let $T$ and $r_k$ be defined as in Theorem \ref{thmain}. Then, the total work required to get an $\epsilon$-accurate solution is
	\begin{align}\label{cgwork}
	\mathcal{O}\left((n + \bar{r}\sfrac{\sigma^2}{\mu^2})d\log\left(\sfrac{1}{\epsilon}\right)\right), \quad
	\mbox{where} \ \  \bar{r} = \max_k r_k.
	\end{align}
\end{cor}

This result involves the quantity $\bar{r}$, which  is the maximum number of CG iterations performed at any iteration of the SSN-CG method.
This quantity  depends  on the eigenvalue distribution of the subsampled Hessians. For various  eigenvalue distributions, such as clustered eigenvalues, $\bar r$ can be much smaller than $d$.
Such favorable eigenvalue distributions  can be observed in many applications, including machine learning problems. 
Figure~\ref{fig:eigs} depicts the  spectrum of $\nabla^2 F_{T}(w^*)$ for 4 data sets used in the logistic regression experiments discussed in the previous section. 

\begin{figure}[]
	\begin{centering}
		\includegraphics[width=1\textwidth]{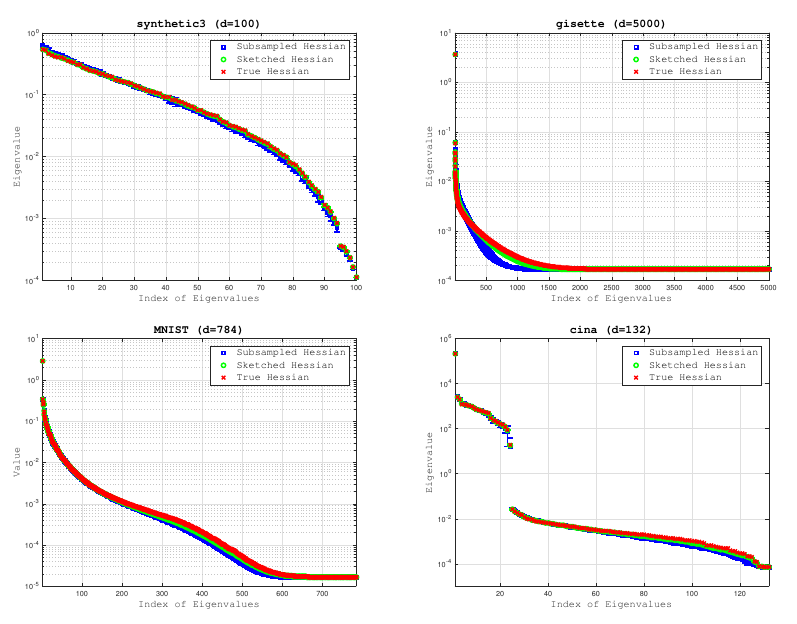}
	\par\end{centering}
	\caption{Distribution of eigenvalues of $\nabla^2 F_{T}(w^*)$, where $F$ is logistic loss with $\ell_2$ regularization \eqref{eq:bin_class},   for 4 data sets: \texttt{synthetic3}, \texttt{gisette}, \texttt{MNIST}, \texttt{cina}. The subsample size $T$ is chosen as $50\%$, $40\%$, $5\%$ and $5\%$ of the  size of the whole data set, respectively.}
	\label{fig:eigs}
\end{figure}
For {\tt gisette, MNIST} and {\tt cina} one can observe that a small number of CG iterations can give rise to a rapid improvement in accuracy in the solution of the linear system. For example, for {\tt cina}, the CG method improves the accuracy in the solution by 2 orders of magnitude during CG iterations 20-24.  For {\tt gisette}, CG exhibits a fast initial increase in accuracy, after which it is not cost effective to continue {performing CG steps} due to the spectral distribution of this problem.
The least favorable case is given by the {\tt synthetic3} data set, which was constructed to have a high condition number and a wide spectrum of eigenvalues. The behavior of CG on  this problem is similar to that predicted by its worst-case analysis. Additional plots are given Appendix \ref{sec:eigs}.

Let us now consider the Newton-Sketch algorithm, evaluate the total work complexity to get an $\epsilon$-accurate solution and compare with the work complexity of the SSN-CG method. Pilanci and Wainwright  \cite[Theorem 1]{pilanci2017newton}  provide a linear-quadratic convergence rate (stated in probability) for the exact Newton-Sketch method,  similar in form to \eqref{twott}-\eqref{twott1}, with Term 2 in \eqref{twott} equal to zero since linear systems are solved exactly. In order to achieve local linear convergence, the user-supplied parameter $\epsilon$ should be $\mathcal{O}({1}/{\kappa})$.  As a result, the sketch dimension $m$ is given by
\begin{align*}
m=\mathcal{O}\left(\sfrac{W^2}{\epsilon^2}\right)= \mathcal{O}\left(\kappa^2 \min\{n,d\}\right),
\end{align*}
where the second equality follows from the fact that the square of the Gaussian width $W$ is at most $\min\{n,d\}$ \cite{pilanci2017newton}. The per-iteration cost of Newton-Sketch with a randomized Hadamard transform is
\begin{align*}
\mathcal{O}\left(nd\log(m) + dm^2\right) = \mathcal{O}\left(nd\log(\kappa d) + d^3\kappa^4\right).
\end{align*}
Hence, the total work complexity required to get an $\epsilon$-accurate solution is  
\begin{align}\label{sketchwork}
\mathcal{O}\left((n + \kappa^4d^2)d\log\left(\sfrac{1}{\epsilon}\right)\right).
\end{align}
Note that in order to compare this with \eqref{cgwork}, we need to bound the parameters $\sigma$ and $\bar{r}$. Using the definitions given in the Assumptions, we have that $\sigma^2$ is bounded by a multiple of $L^2$, and we know that CG converges in at most $d$ iterations, thus $\bar{r} \leq d$. In the most pessimistic case CG requires $d$ iterations, and then \eqref{sketchwork} and \eqref{cgwork} are similar, with Newton-Sketch having a slightly higher cost compared to SSN-CG.

In the analysis stated above, we made certain implicit assumptions the algorithms and the problems. In the SSN method, it was assumed that the number of subsamples is less than the number of examples $n$; and in Newton-Sketch, the sketch dimension is assumed to be less than $n$. This implies that for these methods we made the implicit assumption that the number of data points is large enough so that $n> \kappa^2$.

\section{Final Remarks}
\label{sec:main_findings}

The theoretical properties of Newton-Sketch and subsampled Hessian Newton methods have recently received much attention, but their relative performance on large-scale applications has not been studied in the literature. In this paper, we focus on the finite-sum problem and observed that, although a sketched matrix based on Hadamard matrices has more information than a subsampled Hessian, and approximates the spectrum of the true Hessian more accurately, the two methods perform on par in our tests, as measured by effective gradient evaluations. However, in terms of CPU time, the subsampled Hessian Newton method is much faster than the Newton-Sketch method, which has a high per-iteration cost associated with the construction of the linear system \eqref{sketch}.  A Subsampled Newton method that employs the CG algorithm to solve the linear system (the SSN-CG method) is particularly appealing as one can  coordinate the amount of second order information employed at every iteration with the accuracy of the linear solver, and find an efficient implementation for a given problem. The power of the CG algorithm as an inner iterative solver is evident in our tests and is quantified in the complexity analysis presented in this paper. In particular, we find that CG is a more efficient solver than a first order stochastic gradient iteration.
	 
Our numerical tests show that the two second order methods studied in this paper are generally more efficient than SVRG (a popular first order method), on some logistic regression problems	 

\bibliographystyle{plain}

\bibliography{references}

\newpage

\appendix

\section{Complete Numerical Results}
\label{sec:ext_num}

In this Section, we present further numerical results, on the data sets listed in Tables \ref{tab:synth} and \ref{tab:real}.

\begin{table}[H]
\scriptsize
\centering
\setlength\tabcolsep{4pt}
\begin{minipage}{0.48\textwidth}
\centering
\caption{Synthetic Datasets \cite{mukherjee2013parallel}.}
\begin{tabular}{r|r|r|c}
 \multicolumn{1}{c|}{Dataset}                          &  \multicolumn{1}{c|}{$n$ (train; test)}                  & \multicolumn{1}{c|}{$d$}                      & \multicolumn{1}{c}{$\kappa$}                    \\ \hline\hline
\texttt{synthetic1}                      & (9000; 1000)                        & 100                       & $10^2$             \\
\texttt{synthetic2}                       & (9000; 1000)                        & 100                       & $10^4$               \\ 
\texttt{synthetic3}                       & (90000; 10000)                      & 100                       & $10^4$                \\
\texttt{synthetic4}                       & (90000; 10000)                      & 100                       & $10^5$                  \\ 
\end{tabular}
\label{tab:synth} 
\end{minipage}%
\hfill
\begin{minipage}{0.48\textwidth}
\centering
\caption{Real Datasets.} 
\begin{tabular}{l|r|r|c|c}
\multicolumn{1}{c|}{Dataset} & \multicolumn{1}{c|}{$n$ (train; test)} & \multicolumn{1}{c|}{$d$} & \multicolumn{1}{c|}{$\kappa$} 
&\multicolumn{1}{c}{ Source}
\\ \hline \hline 
\texttt{australian} & (621; 69) & 14 & $10^6$, $10^2$& \cite{CC01a} \\
\texttt{reged} & (450; 50) & 999 & $10^4$&\cite{guyon2008design} \\
\texttt{mushroom} & (5,500; 2,624) & 112 & $10^2$& \cite{CC01a}\\
\texttt{ijcnn1} &  (35,000; 91,701) & 22 & $10^2$&\cite{CC01a}\\
\texttt{cina} & (14,430; 1603) & 132 & $10^9$ &\cite{guyon2008design}\\
\texttt{ISOLET}  & (7,017; 780) & 617 & $10^4$ & \cite{CC01a}\\
\texttt{gisette} & (6,000; 1,000) & 5,000 & $10^4$&\cite{CC01a}\\
\texttt{cov} & (522,911; 58101) & 54 & $10^3$& \cite{CC01a}\\
\texttt{MNIST}  & (60,000; 10,000) & 748 & $10^5$& \cite{lecun1995learning}\\
\texttt{rcv1} & (20,242; 677,399)  & 47,236 & $10^3$&\cite{CC01a}\\
\texttt{real-sim} & (65,078; 7,231)  & 20,958 & $10^3$&\cite{CC01a}\\
\end{tabular} 
 \label{tab:real} 
\end{minipage}
\end{table}

The synthetic data sets were generated randomly as described in \cite{mukherjee2013parallel}. These data sets were created to have different number of samples ($n$) and different condition numbers ($\kappa$), as well as a wide spectrum of eigenvalues (see Figures \ref{fig: eigs_synthetic1}-\ref{fig: eigs_synthetic6}). We also explored the performance of the methods on popular machine learning data sets chosen to have different number of samples ($n$), different number of variables ($d$) and different condition numbers ($\kappa$). We used the testing data sets where provided. For data sets where a testing set was not provided, we randomly split the data sets ($90\%, 10\%$) for training and testing, respectively.

We focus on logistic regression classification; the objective function is given by 
 \begin{align} 	\label{eq:bin_class1}
  \min_{w \in \mathbb{R}^d} F(w) = \frac{1}{n}\sum_{i=1}^{n}\log(1+e^{-y^i(w^Tx^i)}) + \frac{\lambda}{2} \|w\|^2,
\end{align}
where $ (x^i, y^i)_{i=1}^n$  denote the training examples and $\lambda = \frac1n$
is the regularization parameter. 

For the experiments in Sections \ref{sec:performance_synth}, \ref{sec:performance} and \ref{sec:sensitivity} (Figures \ref{fig: best_synthetic_app1}-\ref{fig: sensitivity4_app}), we ran four methods: SVRG, Newton-Sketch, SSN-SGI and SSN-CG. The implementation details of the methods are given in Sections \ref{sec:pract_methods} and \ref{sec:num_study}. In Section \ref{sec:performance_synth} we show the performance of the methods on the synthetic data sets, in Section \ref{sec:performance} we show the performance of the methods on popular machine learning data sets and in Section \ref{sec:sensitivity} we examine the sensitivity of the methods on the choice of the hyper-parameters. We report training error vs. iterations, and training error vs. effective gradient evaluations (defined as the sum of all function evaluations, gradient evaluations and Hessian-vector products). In Sections \ref{sec:performance_synth} and \ref{sec:performance} we also report testing loss vs. effective gradient evaluations.


\newpage

\subsection{Performance of methods - Synthetic Data sets}
\label{sec:performance_synth}

Figure \ref{fig: best_synthetic_app1} illustrates the performance of the methods on four synthetic data sets. These problems were constructed to have high condition numbers $10^2 - 10^5$, which is the setting where Newton methods show their strength compared to a first order methods such as SVRG, and to have wide spectrum of eigenvalues, which is the setting that is unfavorable for the CG method.

As is clear from Figure \ref{fig: best_synthetic_app1}, the Newton-Sketch and SSN-CG methods outperform the SVRG and SSN-SGI methods. This is expected due to the ill-conditioning of the problems. The Newton-Sketch method performs slightly better than the SSN-CG method; this can be attributed, primarily, to the fact that the component function in these problems are highly dissimilar. The Hessian approximations constructed by the Newton-Sketch method, use information from all component functions, and thus, better approximate the true Hessian. It is interesting to observe that in the initial stage of these experiments, the SVRG method outperforms the second order methods, however, the progress made by the first order method quickly stagnates.

 
\begin{figure}[H]
\begin{centering}
\includegraphics[width=1\textwidth]{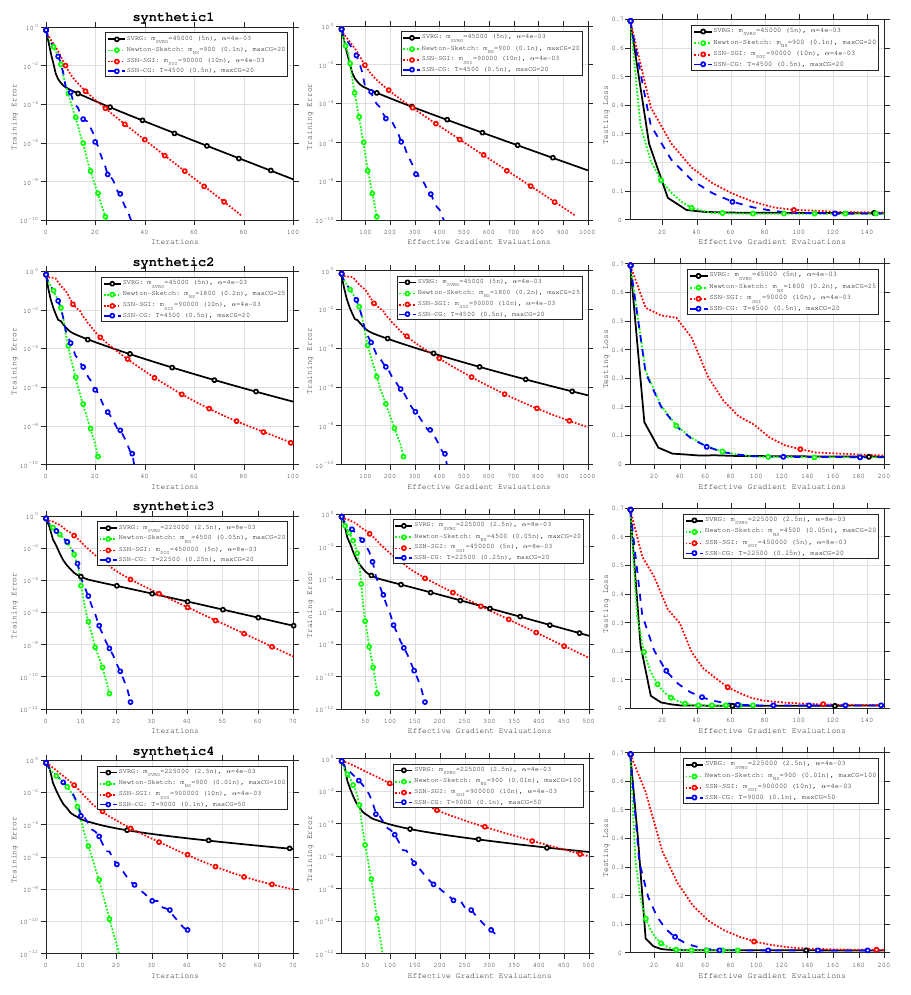}
\par\end{centering}
\caption{Comparison of SSN-SGI, SSN-CG, Newton-Sketch and SVRG on four synthetic data sets: \texttt{synthetic1}; \texttt{synthetic2}; \texttt{synthetic3}; \texttt{synthetic4}. Each row corresponds to a different problem. The first column reports training error ($F(w_k) - F^*$) vs iteration, and the second column reports training error vs effective gradient evaluations, and the third column reports testing loss vs effective gradient evaluations. The hyper-parameters of each method were tuned to achieve the best training error in terms of effective gradient evaluations. The cost of forming the linear systems in the Newton-Sketch method is ignored.}
\label{fig: best_synthetic_app1}
\end{figure}


\subsection{Performance of methods - Machine Learning Data sets}
\label{sec:performance}

Figures \ref{fig: best_real1_app}--\ref{fig: best_real5_app} illustrate the performance of the methods on 12 popular machine learning data sets. As mentioned in Section \ref{results}, the performance of the methods is highly dependent on the problem characteristics. On the one hand, these figures show that there exists an SVRG \emph{sweet-spot}, a regime in which the SVRG method outperforms all stochastic second order methods investigated in this paper; however, there are other problems for which it is efficient to incorporate some form of curvature information in order to avoid stagnation due to ill-conditioning. For such problems, the SVRG method makes slow progress due to the need for a very small step length, which is required to ensure that the method does not diverge.

Note: We did not run Newton-Sketch on the \texttt{cov} data set (Figure \ref{fig: best_real5_app}). This was due to the fact that the number of data points ($n$) in this data set is  large, and so it was prohibitive (in terms of memory) to create the padded sketch matrices.

\begin{figure}[H]
\begin{centering}
\includegraphics[width=1\textwidth]{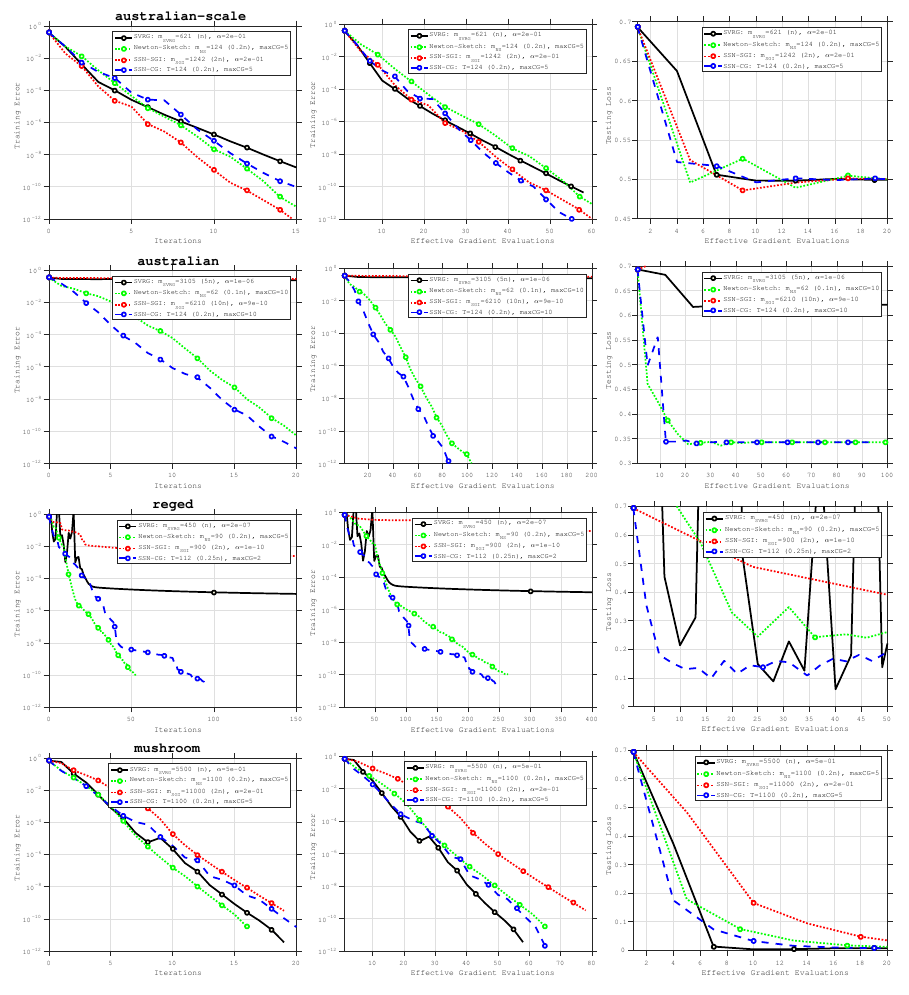}
\par\end{centering}
\caption{Comparison of SSN-SGI, SSN-CG, Newton-Sketch and SVRG on four machine learning data sets: \texttt{australian-scale}; \texttt{australian}; \texttt{reged}; \texttt{mushroom}. Each row corresponds to a different problem. The first column reports training error ($F(w_k) - F^*$) vs iteration, and the second column reports training error vs effective gradient evaluations, and the third column reports testing loss vs effective gradient evaluations. The hyper-parameters of each method were tuned to achieve the best training error in terms of effective gradient evaluations. The cost of forming the linear systems in the Newton-Sketch method is ignored.}
\label{fig: best_real1_app}
\end{figure}


\begin{figure}[H]
\begin{centering}
\includegraphics[width=1\textwidth]{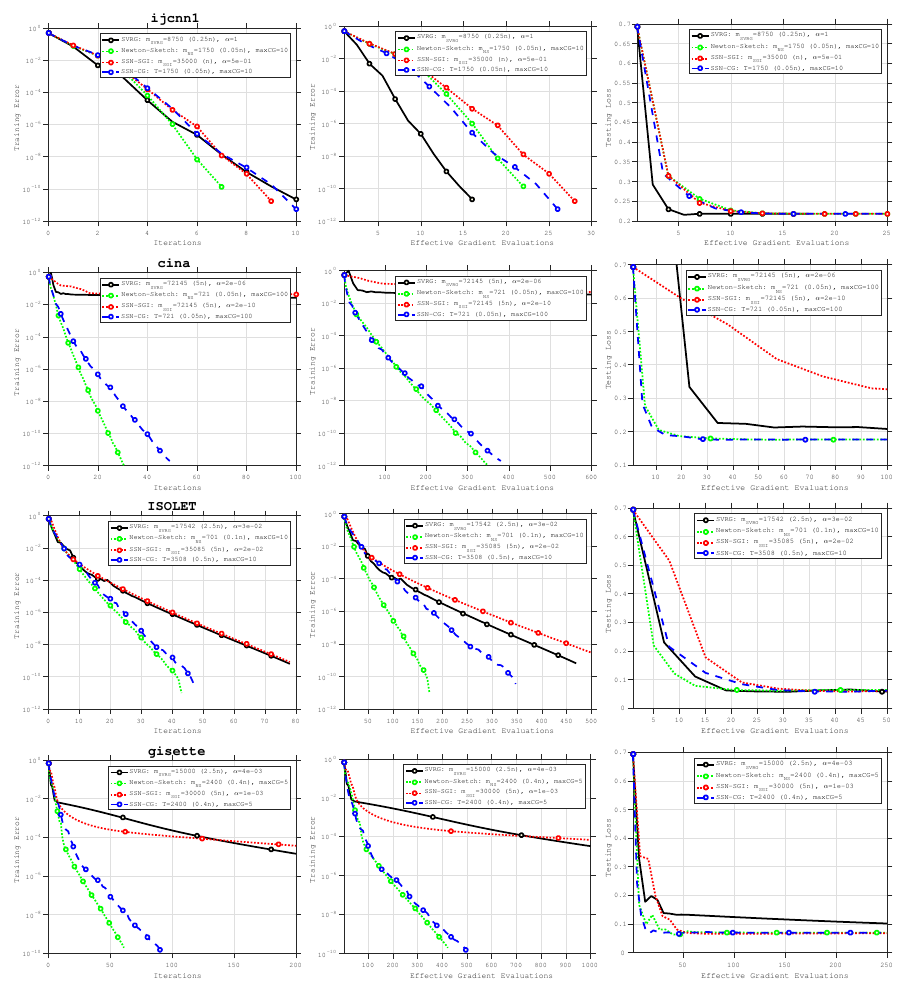}
\par\end{centering}
\caption{Comparison of SSN-SGI, SSN-CG, Newton-Sketch and SVRG on four machine learning data sets: \texttt{ijcnn1}; \texttt{cina}; \texttt{ISOLET}; \texttt{gisette}. Each row corresponds to a different problem. The first column reports training error ($F(w_k) - F^*$) vs iteration, and the second column reports training error vs effective gradient evaluations, and the third column reports testing loss vs effective gradient evaluations. The hyper-parameters of each method were tuned to achieve the best training error in terms of effective gradient evaluations. The cost of forming the linear systems in the Newton-Sketch method is ignored.}
\label{fig: best_real3_app}
\end{figure}


\begin{figure}[H]
\begin{centering}
\includegraphics[width=1\textwidth]{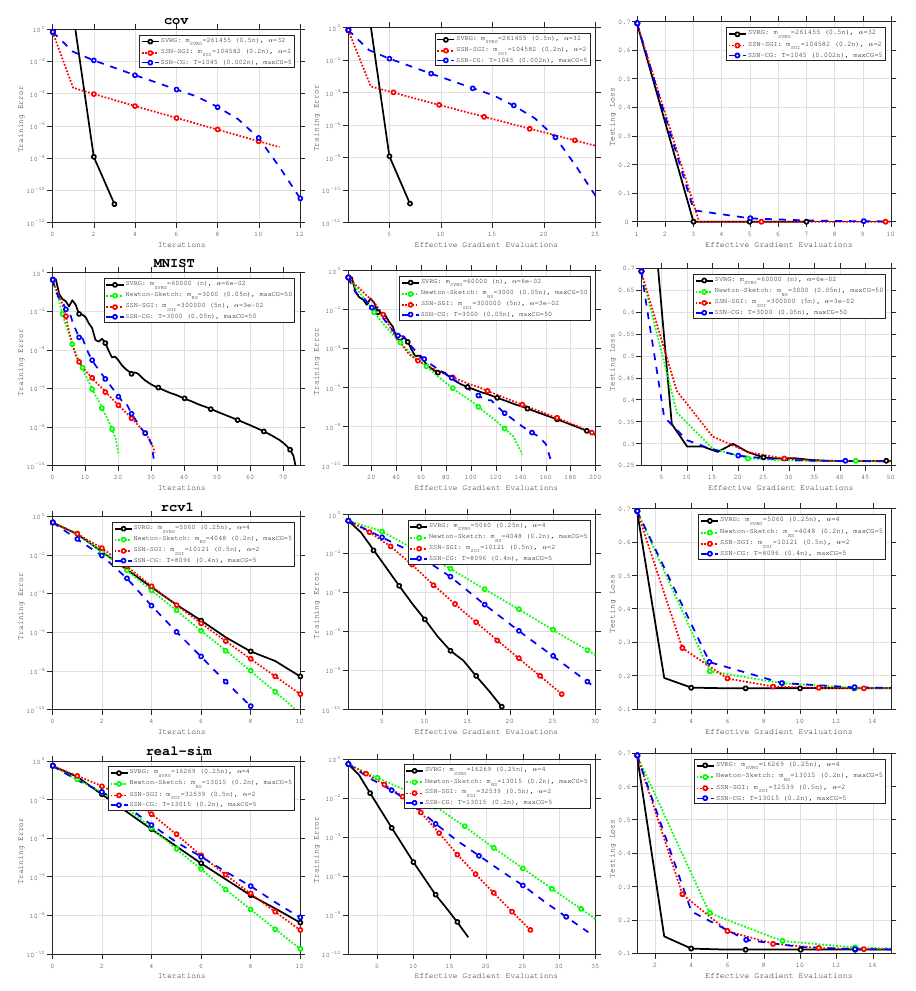}
\par\end{centering}
\caption{Comparison of SSN-SGI, SSN-CG, Newton-Sketch and SVRG on four machine learning data sets: \texttt{cov}; \texttt{MNIST}; \texttt{rcv1}; \texttt{real-sim}. Each row corresponds to a different problem. The first column reports training error ($F(w_k) - F^*$) vs iteration, and the second column reports training error vs effective gradient evaluations, and the third column reports testing loss vs effective gradient evaluations. The hyper-parameters of each method were tuned to achieve the best training error in terms of effective gradient evaluations. The cost of forming the linear systems in the Newton-Sketch method is ignored.}
\label{fig: best_real5_app}
\end{figure}

\newpage
\subsection{Sensitivity of methods}
\label{sec:sensitivity}

In Figures \ref{fig: sensitivity1_app}--\ref{fig: sensitivity4_app} we illustrate the sensitivity of the methods to the choice of hyper-parameters for 4 different data sets: \texttt{synthetic1}, \texttt{australian-scale}, \texttt{mushroom} and \texttt{ijcnn}. One needs to be careful when interpreting the sensitivity results. It appears that all methods have similar sensitivity to the chosen hyper-parameters, however, we should note that this is not the case. More specifically, if the step length $\alpha$ is not chosen appropriately in the SVRG and SSN-SGI methods, these methods can diverge ($\alpha$ too large) or make very slow progress ($\alpha$ too small). We have excluded these runs from the figures presented in this section. Empirically, we observe that the Newton-Sketch and SSN-CG methods are more robust to the choice of the hyper-parameters, and easier to tune. For almost all choices of $m_{\text{NS}}$ and $T$, respectively, and $max_{\text{CG}}$ the methods converge, with the choice of hyper-parameters only affecting the speed of convergence.

\begin{figure}[H]
\begin{centering}
\includegraphics[width=0.76\textwidth]{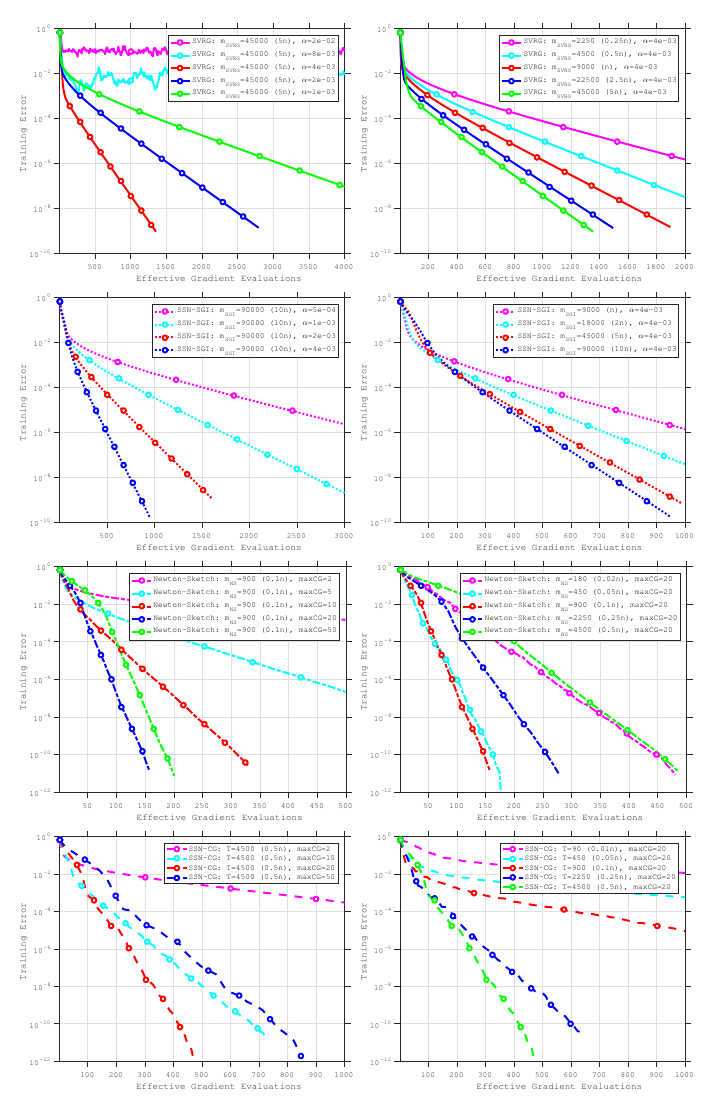}
\par\end{centering}
\caption{\texttt{synthetic1}: Sensitivity of SVRG, SSN-SGI, Newton-Sketch and SSN-CG to the choice of hyper-parameters. Each row depicts a different method. The first column shows the performance of a given method when varying the first hyper-parameter, while the second column shows the performance when varying the second hyper-parameter. Note, for SVRG and Newton-SGI we only illustrate the performance of the methods in hyper-parameter regimes in which the methods did not diverge. }
\label{fig: sensitivity1_app}
\end{figure}

\begin{figure}[H]
\begin{centering}
\includegraphics[width=0.76\textwidth]{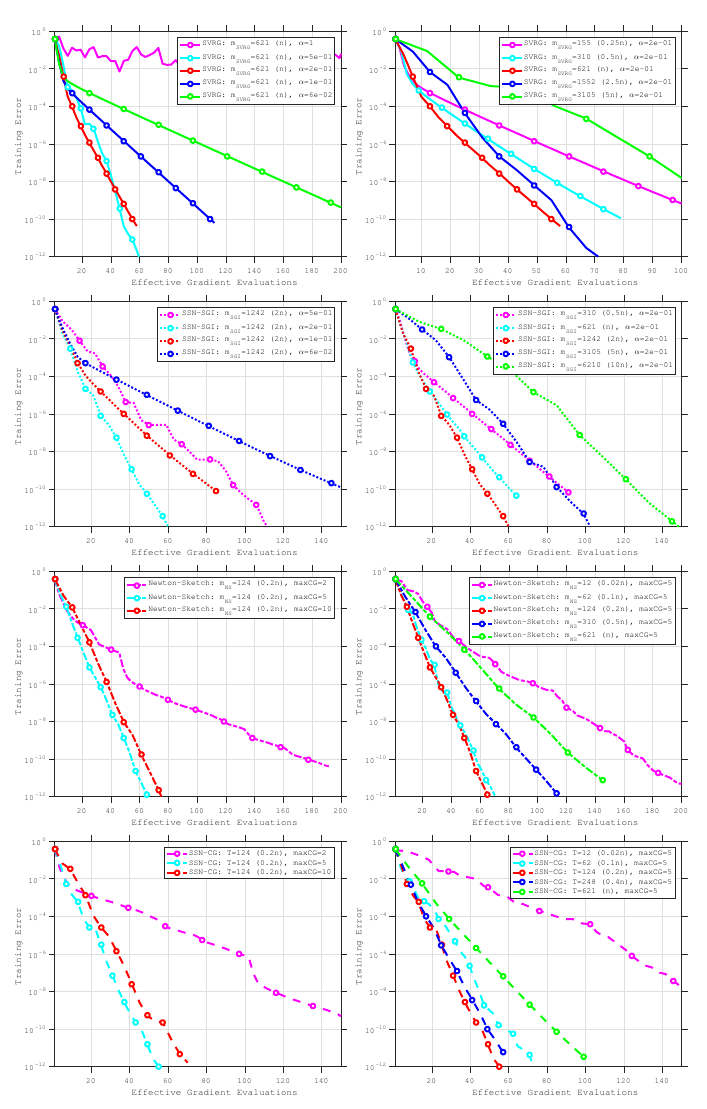}
\par\end{centering}
\caption{\texttt{australian-scale}: Sensitivity of SVRG, SSN-SGI, Newton-Sketch and SSN-CG to the choice of hyper-parameters. Each row depicts a different method. The first column shows the performance of a given method when varying the first hyper-parameter, while the second column shows the performance when varying the second hyper-parameter. Note, for SVRG and Newton-SGI we only illustrate the performance of the methods in hyper-parameter regimes in which the methods did not diverge. }
\label{fig: sensitivity2_app}
\end{figure}

\begin{figure}[H]
\begin{centering}
	\includegraphics[width=0.76\textwidth]{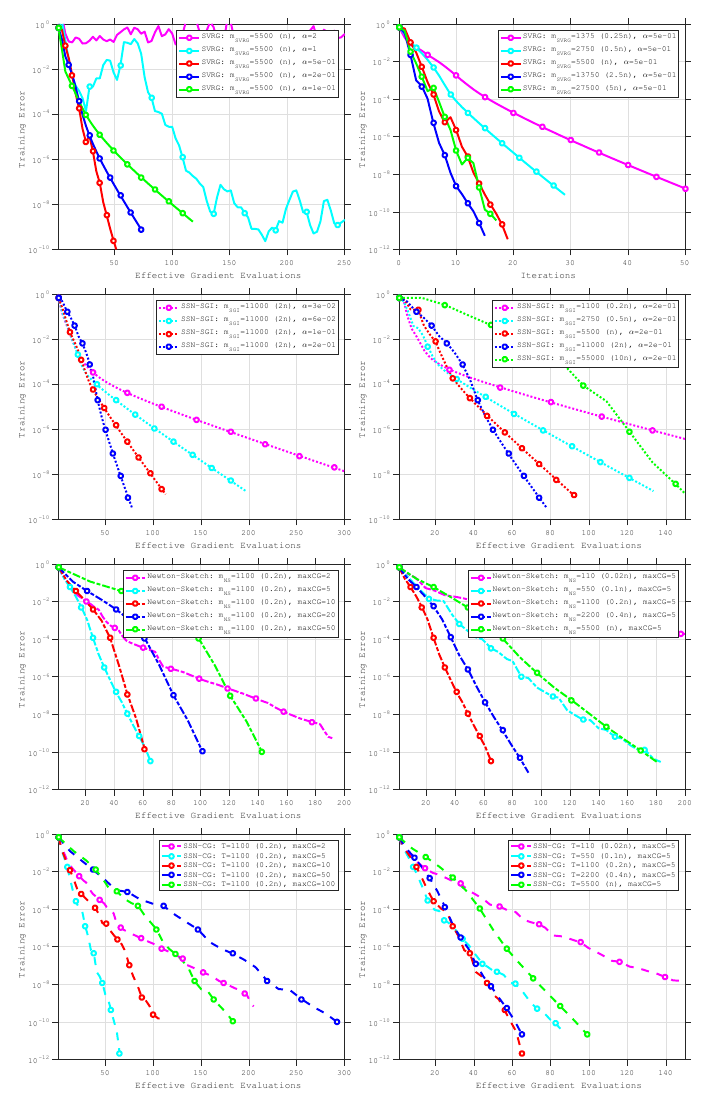}
\par\end{centering}
\caption{\texttt{mushroom}: Sensitivity of SVRG, SSN-SGI, Newton-Sketch and SSN-CG to the choice of hyper-parameters. Each row depicts a different method. The first column shows the performance of a given method when varying the first hyper-parameter, while the second column shows the performance when varying the second hyper-parameter. Note, for SVRG and Newton-SGI we only illustrate the performance of the methods in hyper-parameter regimes in which the methods did not diverge. }
\label{fig: sensitivity3_app}
\end{figure}

\begin{figure}[H]
\begin{centering}
\includegraphics[width=0.76\textwidth]{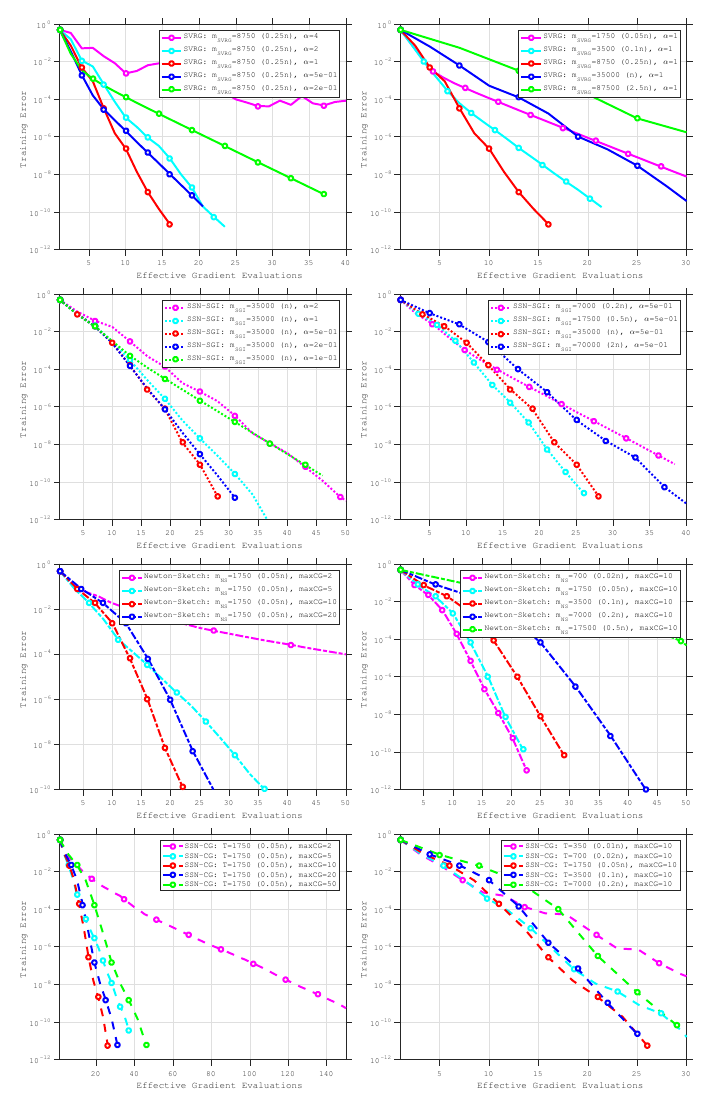}
\par\end{centering}
\caption{\texttt{ijcnn1}: Sensitivity of SVRG, SSN-SGI, Newton-Sketch and SSN-CG to the choice of hyper-parameters. Each row depicts a different method. The first column shows the performance of a given method when varying the first hyper-parameter, while the second column shows the performance when varying the second hyper-parameter. Note, for SVRG and Newton-SGI we only illustrate the performance of the methods in hyper-parameter regimes in which the methods did not diverge. }
\label{fig: sensitivity4_app}
\end{figure}

\newpage
\section{Eigenvalues}
\label{sec:eigs}

In this Section, we present the eigenvalue spectrum for different data sets and different subsample sizes. Here $T$ denotes the number of subsamples used in the subsampled Hessian and $m_{\text{NS}}$ denotes the number of rows of the Sketch matrix. In order to make a fair comparison, we chose $T = m_{\text{NS}}$ for each figure.

To calculate the eigenvalues we used the following procedure. For each data set and subsample size we: 
\begin{enumerate}
	\item Computed the true Hessian 
		\begin{align*}		
			\nabla^2 F(w^*) = \frac{1}{n} \sum_{i}^n \nabla^2 F_i(w^*) ,
		\end{align*}
	 of \eqref{eq:bin_class} at $w^*$, and computed the eigenvalues of the true Hessian matrix (red).
	\item Computed 10 different subsampled Hessians of \eqref{eq:bin_class} at the optimal point $w^*$ using different samples $T_j\subset \{1,2,\ldots,n \}$ for $j=1,2,\ldots,10$ ($T_j = |T|$)
		\begin{align*}		
			\nabla^2 F_{T_j}(w^*) = \frac{1}{\left| T_j \right|} \sum_{i \in T_j} \nabla^2 F_i(w^*) ,
		\end{align*}
	 computed the eigenvalues of each subsampled Hessian and took the average of the eigenvalues across the 10 replications (blue).
	\item Computed 10 different sketched Hessians of \eqref{eq:bin_class} at the optimal point $w^*$ using different sketch matrices $S^j \in \mathbb{R}^{m_{\text{NS}} \times n}$ for $j=1,2,\ldots,10$
	\begin{align*}		
			\nabla^2 F_{S^j}(w^*) =   \Big( \big((S^j)\nabla^2 F(w^*)^{\sfrac{1}{2}}\big)^TS^j\nabla^2 F(w^*)^{\sfrac{1}{2}}\Big),
		\end{align*}
	computed the eigenvalues of each sketched Hessian and took the average of the eigenvalues across the 10 replications (green).
\end{enumerate}
All eigenvalues were computed in MATLAB using the function \texttt{eig}. Figures \ref{fig: eigs_synthetic1}--\ref{fig: MNIST} show the distribution of the eigenvalues for different data sets. Each figure represents one data set and depicts the eigenvalue distribution for three different subsample and sketch sizes. The red marks, blue squares and green circles represent the eigenvalues of the true Hessian, the average eigenvalues of the subsampled Hessians and the average eigenvalues of the sketched Hessians, respectively. Since we computed 10 different subsampled and sketched Hessians, we also show the upper and lower bounds of each eigenvalue with error bars.

As is clear from the figures below, contingent on a reasonable subsample size and sketch size, the eigenvalue spectra of the subsampled and sketched Hessians are able to capture the eigenvalue spectrum of the true Hessian. It appears however, that the eigenvalue spectra of the sketched Hessians are closer to the spectra of the true Hessian. By this we mean both that with fewer rows of the sketch matrix, the approximations of the average eigenvalues of the sketched Hessians are closer to the true eigenvalues of the Hessian, and that the error bars of the sketched Hessians are smaller than those of the subsampled Hessians. This is not surprising as the sketched Hessians use information from all components functions whereas the subsampled Hessians are constructed from a subset of the component functions. This is interesting in itself, however, the main reasons we present these results is to demonstrate that the eigenvalue distributions of machine learning problems appear to have favorable distributions for CG.

\begin{figure}[H]
\begin{centering}
\includegraphics[width=0.9\textwidth]{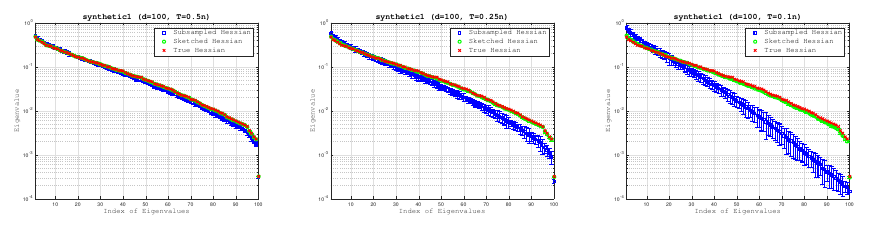}
\par\end{centering}
\caption{Distribution of eigenvalues. \texttt{synthetic1}:  $d = 100$; $n = 9,000$; $T = 0.5n, 0.25n, 0.1n$}
\label{fig: eigs_synthetic1}
\end{figure}

\begin{figure}[H]
\begin{centering}
\includegraphics[width=0.9\textwidth]{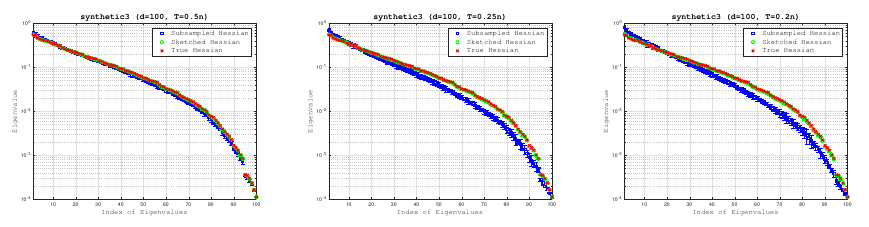}
\par\end{centering}
\caption{Distribution of eigenvalues. \texttt{synthetic2}:  $d = 100$; $n = 9,000$; $T = 0.5n, 0.25n, 0.2n$}
\label{fig: eigs_synthetic3}
\end{figure}

\begin{figure}[H]
\begin{centering}
\includegraphics[width=0.9\textwidth]{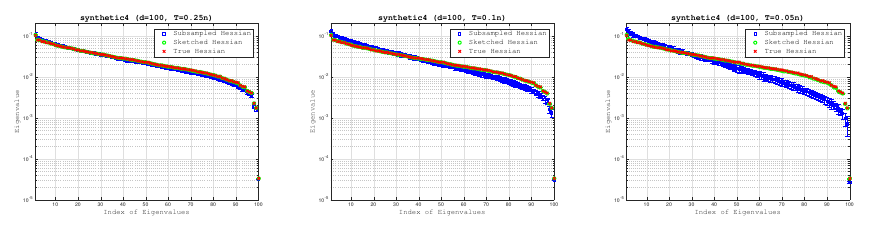}
\par\end{centering}
\caption{Distribution of eigenvalues. \texttt{synthetic3}: $d = 100$; $n = 90,000$; $T = 0.25n, 0.1n, 0.05n$}
\label{fig: eigs_synthetic4}
\end{figure}

\begin{figure}[H]
\begin{centering}
\includegraphics[width=0.9\textwidth]{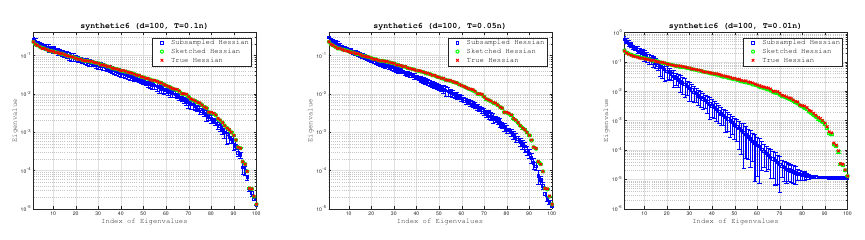}
\par\end{centering}
\caption{Distribution of eigenvalues. \texttt{synthetic4}:  $d = 100$; $n = 90,000$; $T = 0.1n, 0.05n, 0.01n$}
\label{fig: eigs_synthetic6}
\end{figure}

\begin{figure}[H]
\begin{centering}
\includegraphics[width=0.9\textwidth]{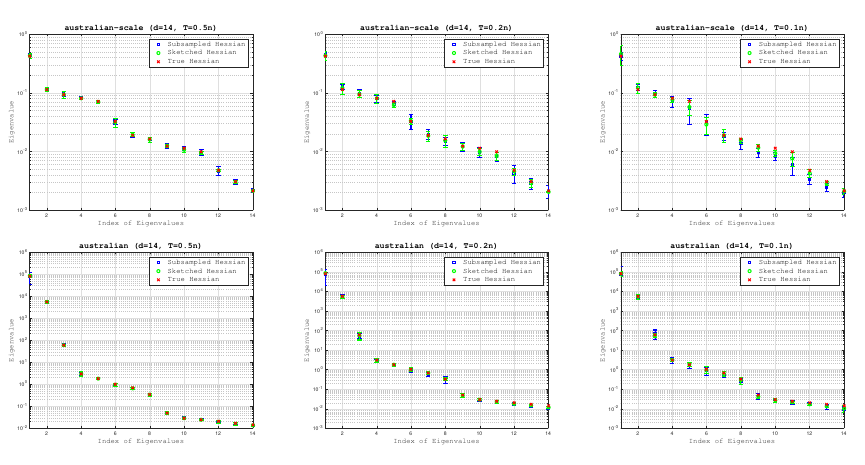}
\par\end{centering}
\caption{Distribution of eigenvalues. \texttt{australian\_scale}, \texttt{australian}: $d = 14$; $n = 621$; $T = 0.5n, 0.2n, 0.1n$}
\label{fig: australian}
\end{figure}

\begin{figure}[H]
\begin{centering}
	\includegraphics[width=0.9\textwidth]{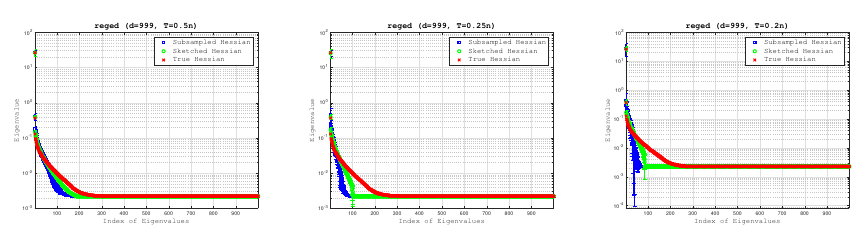}
\par\end{centering}
\caption{Distribution of eigenvalues. \texttt{reged}: $d = 999$; $n = 450$; $T = 0.5n, 0.25n, 0.2n$}
\label{fig: reged}
\end{figure}

\begin{figure}[H]
\begin{centering}
\includegraphics[width=0.9\textwidth]{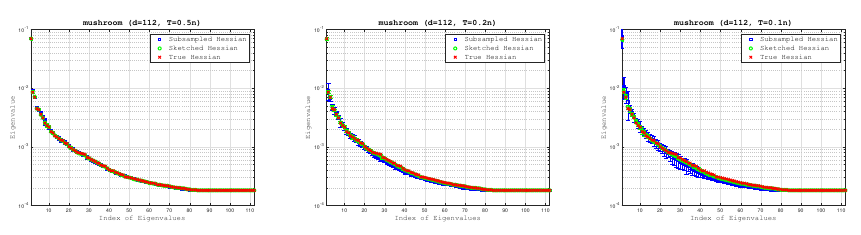}
\par\end{centering}
\caption{Distribution of eigenvalues. \texttt{mushroom}: $d = 112$; $n = 5500$; $T = 0.5n, 0.2n, 0.1n$}
\label{fig: mushroom}
\end{figure}

\begin{figure}[H]
\begin{centering}
\includegraphics[width=0.9\textwidth]{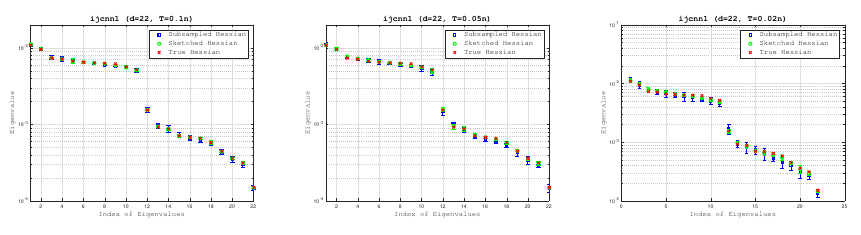}
\par\end{centering}
\caption{Distribution of eigenvalues. \texttt{ijcnn}: $d = 22$; $n = 35,000$; $T = 0.1n, 0.05n, 0.02n$}
\label{fig: ijcnn}
\end{figure}

\begin{figure}[H]
\begin{centering}
\includegraphics[width=0.9\textwidth]{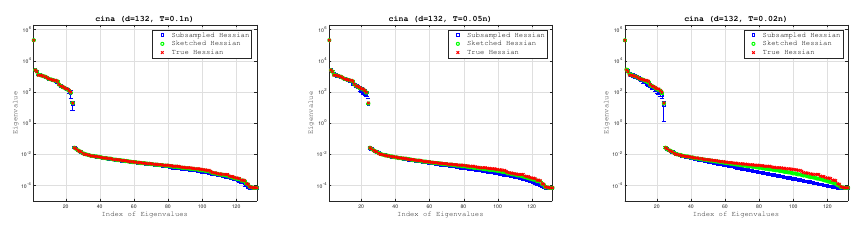}
\par\end{centering}
\caption{Distribution of eigenvalues. \texttt{cina}: $d = 132$; $n = 14,430$; $T = 0.1n, 0.05n, 0.02n$}
\label{fig: cina}
\end{figure}

\begin{figure}[H]
\begin{centering}
\includegraphics[width=0.9\textwidth]{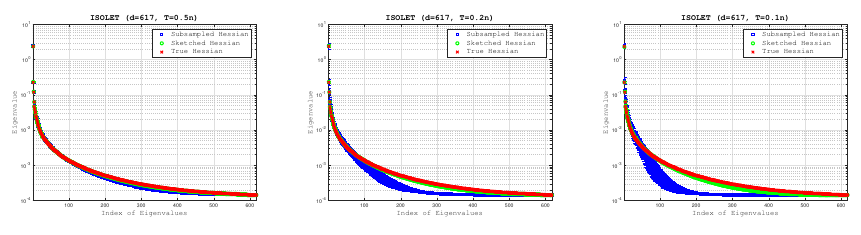}
\par\end{centering}
\caption{Distribution of eigenvalues. \texttt{ISOLET}: $d = 617$; $n = 7,017$; $T = 0.5n, 0.2n, 0.1n$}
\label{fig: ISOLET}
\end{figure}

\begin{figure}[H]
\begin{centering}
\includegraphics[width=0.9\textwidth]{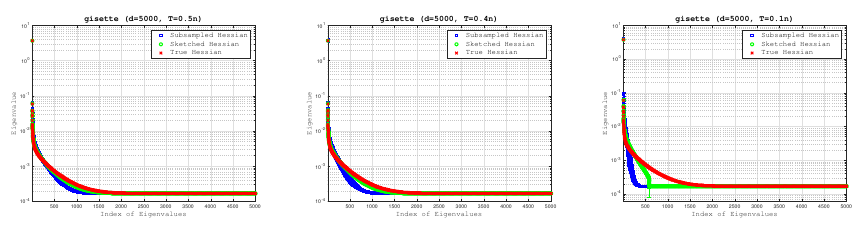}
\par\end{centering}
\caption{Distribution of eigenvalues. \texttt{gisette}: $d = 5,000$; $n = 6,000$; $T = 0.5n, 0.4n, 0.1n$}
\label{fig: gisette}
\end{figure}

\begin{figure}[H]
\begin{centering}
\includegraphics[width=0.9\textwidth]{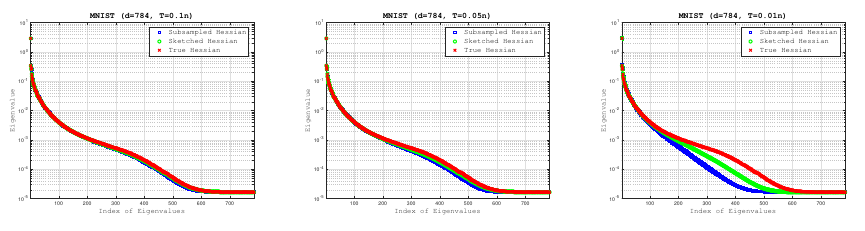}
\par\end{centering}
\caption{Distribution of eigenvalues. \texttt{MNIST}: $d = 784$; $n = 60,000$; $T = 0.1n, 0.05n, 0.02n$}
\label{fig: MNIST}
\end{figure}

\end{document}